\documentclass[a4paper, reqno, 12pt]{amsart}
\usepackage{stmaryrd}
\usepackage{float}

\usepackage[usenames,dvipsnames]{color}
\usepackage{amsthm,amsfonts,amssymb,amsmath,amsxtra}
\usepackage{tikz}
\usepackage{tkz-euclide}
\usepackage[all]{xy}
\SelectTips{cm}{}
\usepackage{xr-hyper}
\usepackage[colorlinks=
citecolor=Black,
linkcolor=Red,
urlcolor=Blue]{hyperref}
\usepackage{verbatim}

\usepackage{geometry}
\usepackage{mathrsfs}

\RequirePackage{xspace}
\RequirePackage{etoolbox}
\RequirePackage{varwidth}
\RequirePackage{enumitem}
\RequirePackage{tensor}
\RequirePackage{mathtools}
\RequirePackage{longtable}
\RequirePackage{multirow}

\mathtoolsset{showonlyrefs}

\def\le{\leqslant}

\def\d{\delta}

\def\i{^{-1}}

\def\<{\langle}
\def\>{\rangle}

\newcommand{\BF}{\ensuremath{\mathbb {F}}\xspace}
\newcommand{{\BG}}{\ensuremath{\mathbb {G}}\xspace}

\newcommand{{\BK}}{\ensuremath{\mathbb {K}}\xspace}

\newcommand{\BN}{\ensuremath{\mathbb {N}}\xspace}

\newcommand{\BQ}{\ensuremath{\mathbb {Q}}\xspace}

\newcommand{\BS}{\ensuremath{\mathbb {S}}\xspace}

\newcommand{\BZ}{\ensuremath{\mathbb {Z}}\xspace}

\newcommand{\CH}{\ensuremath{\mathcal {H}}\xspace}

\newcommand{\CY}{\ensuremath{\mathcal {Y}}\xspace}

\DeclareMathOperator{\charac}{char}

\DeclareMathOperator{\Spec}{Spec}

\DeclareMathOperator{\supp}{supp}

\newtheorem{theorem}{Theorem}
\newtheorem{proposition}[theorem]{Proposition}
\newtheorem{lemma}[theorem]{Lemma}

\theoremstyle{definition}
\newtheorem{definition}[theorem]{Definition}

\newtheorem{remark}[theorem]{Remark}

\newtheoremstyle{query}%
{}{}
{\color{red}}
{}
{\sffamily\bfseries}{:}{12pt}
{}
\theoremstyle{query}
\newtheorem{aq}{Author Query/Comment}

\newcommand{\baq}{\begin{aq}}
\newcommand{\eaq}{\end{aq}}

\numberwithin{equation}{section}
\numberwithin{theorem}{section}

\usepackage{lineno}

\setitemize[0]{leftmargin=*,itemsep=\the\smallskipamount}
\setenumerate[0]{leftmargin=*,itemsep=\the\smallskipamount}

\renewcommand{\to}{%
   \ifbool{@display}{\longrightarrow}{\rightarrow}%
   }
\let\shortmapsto\mapsto
\renewcommand{\mapsto}{%
   \ifbool{@display}{\longmapsto}{\shortmapsto}%
   }
\newlength{\olen}
\newlength{\ulen}
\newlength{\xlen}
\newcommand{\xra}[2][]{%
   \ifbool{@display}%
      {\settowidth{\olen}{$\overset{#2}{\longrightarrow}$}%
       \settowidth{\ulen}{$\underset{#1}{\longrightarrow}$}%
       \settowidth{\xlen}{$\xrightarrow[#1]{#2}$}%
       \ifdimgreater{\olen}{\xlen}%
          {\underset{#1}{\overset{#2}{\longrightarrow}}}%
          {\ifdimgreater{\ulen}{\xlen}%
             {\underset{#1}{\overset{#2}{\longrightarrow}}}
             {\xrightarrow[#1]{#2}}}}%
      {\xrightarrow[#1]{#2}}
   }
\makeatother
\newcommand{\xyra}[2][]{%
   \settowidth{\xlen}{$\xrightarrow[#1]{#2}$}%
   \ifbool{@display}%
      {\settowidth{\olen}{$\overset{#2}{\longrightarrow}$}%
       \settowidth{\ulen}{$\underset{#1}{\longrightarrow}$}%
       \ifdimgreater{\olen}{\xlen}%
          {\mathrel{\xymatrix@M=.12ex@C=3.2ex{\ar[r]^-{#2}_-{#1} &}}}%
          {\ifdimgreater{\ulen}{\xlen}%
             {\mathrel{\xymatrix@M=.12ex@C=3.2ex{\ar[r]^-{#2}_-{#1} &}}}
             {\mathrel{\xymatrix@M=.12ex@C=\the\xlen{\ar[r]^-{#2}_-{#1} &}}}}}%
      {\mathrel{\xymatrix@M=.12ex@C=\the\xlen{\ar[r]^-{#2}_-{#1} &}}}%
   }
\makeatletter
\newcommand{\xla}[2][]{%
   \ifbool{@display}%
      {\settowidth{\olen}{$\overset{#2}{\longleftarrow}$}%
       \settowidth{\ulen}{$\underset{#1}{\longleftarrow}$}%
       \settowidth{\xlen}{$\xleftarrow[#1]{#2}$}%
       \ifdimgreater{\olen}{\xlen}%
          {\underset{#1}{\overset{#2}{\longleftarrow}}}%
          {\ifdimgreater{\ulen}{\xlen}%
             {\underset{#1}{\overset{#2}{\longleftarrow}}}
             {\xleftarrow[#1]{#2}}}}%
      {\xleftarrow[#1]{#2}}
   }
\newcommand{\isoarrow}{%
   \ifbool{@display}{\overset{\sim}{\longrightarrow}}{\xrightarrow\sim}%
   }

\newcommand{\sfK}{\ensuremath{\mathsf{K}}\xspace}
\newcommand{\sfk}{\ensuremath{\mathsf{k}}\xspace}

\begin{document}

\title{Lusztig varieties for regular elements}
\author{Xuhua He}
\address{Department of Mathematics and New Cornerstone Science Laboratory, The University of Hong Kong, Pokfulam, Hong Kong, Hong Kong SAR, China}
\email{xuhuahe@hku.hk}
\author{Ruben La}
\address{Department of Mathematics and New Cornerstone Science Laboratory, The University of Hong Kong, Pokfulam, Hong Kong, Hong Kong SAR, China}
\email{rubenla@hku.hk}

\begin{abstract}
Let $G$ be a connected reductive group over an algebraically closed field. Let $B$ be a Borel subgroup of $G$ and $W$ be the associated Weyl group. We show that for any $w \in W$ that is not contained in any standard parabolic subgroup of $W$, the intersection of the Bruhat cell $B w B$ with any regular conjugacy class of $G$ is always irreducible. We then prove that the associated Lusztig varieties are irreducible. This extends the previous work of Kim \cite{kim2020homology} on the regular semisimple and regular unipotent elements. The irreducibilitiy result of Lusztig varieties will be used in an upcoming work in the study of affine Lusztig varieties. 

\end{abstract}
\maketitle

\tableofcontents

\section{Introduction}

Let $\sfK$ be an algebraically closed field and $\BG$ be a connected reductive group defined over $\sfK$. Let $G = \BG(\sfK)$ be the group of $\sfK$-points. Let $\d$ be an automorphism of the Dynkin diagram of $\BG$. Let $B$ be a $\d$-stable Borel subgroup and $T \subseteq B$ be a $\d$-stable maximal torus of $G$. Let $W=N_G(T)/T$ be the associated Weyl group and $\BS$ be the set of simple reflections determined by $B$. For any $w \in W$, we fix a representative $\dot w \in N_G(T)$. 
For $J \subseteq \BS$, let $W_J \subseteq W$ be the parabolic subgroup generated by the simple reflections $\{s_i\colon i \in J\}$. For $w \in W$, let $\text{supp}(w)$ be the subset of $\BS$ consisting of all simple reflections occurring in some reduced expression of $w$ and let $\text{supp}_\delta(w) = \bigcup_{i\in\BN} \text{supp}(\delta^i(w))$. Note that $\text{supp}_\d(w)$ is the smallest $\d$-stable subset $J$ of $\BS$ such that $w \in W_J$. 

For $g,h \in G$, consider the $\d$-conjugation action defined by $g\cdot_\delta h = gh\delta(g^{-1})$. The $\d$-centralizer of $h \in G$ is defined to be $Z_{G, \d}(h)=\{g \in G\colon g \cdot_\d h=h\}$. The Lusztig variety associated with $w \in W$ and $h \in G$ is defined as 
$$\CY_{w,h,\d}^G=\CY_{w, h, \d}=\{g B \in G/B\colon g \i h \d(g) \in B \dot w B\}.$$
The isomorphism class of a Lusztig variety does not depend on the choice of $\delta$-invariant Borel subgroup $B$, and further only depends on the $\delta$-conjugacy class of $h$. 

Lusztig varieties were originally introduced in \cite[\S1]{lusztig197912} for the case where $\delta = 1$ and $h$ is a regular semisimple element. For $h$ an arbitrary element in $G$, they were introduced in \cite{lusztig1985character1}, where they played an important role in the definition of character sheaves. 

An element $h \in G$ is called {\it $\delta$-regular} if $Z_{G, \d}(h)$ is of minimal dimension among all the elements $h \in G$. In this case, $\dim Z_{G, \d}(h)= \dim T^\delta$. We call the $\d$-conjugacy class $G \cdot_\d h$ {\it regular} if it contains a $\d$-regular element. 

The main result of this paper is the following. 
\begin{theorem}
    Suppose $\supp_\delta(w) =\BS$ and let $h$ be a $\d$-regular element in $G$. Then $\CY_{w, h, \d}$ is irreducible.
\end{theorem}
For $\delta = 1$, the case where $h$ is regular semisimple was settled by Kim, who gave a formula for the number of irreducible components in \cite[Theorem 4.5]{kim2020homology}. The case that $h$ is unipotent and regular, with the additional assumption that the characteristic of $\sfK$ is good for $\BG$, was also settled by Kim in \cite[Proposition 5.6]{kim2020homology}, by showing that the Lusztig variety is always irreducible for any choice of $w \in W$.

\subsection{Acknowledgements}
The authors would like to thank Cheng-Chiang Tsai for carefully explaining \cite[Appendix B]{tsai2020components} and the reduction to the closure of the finite field case in Theorem \ref{theorem:lusztigvarietyirreduciblearbitrarycharacteristic}. We also thank George Lusztig and Kenneth Chiu for helpful discussions. XH is partially supported by the New Cornerstone Science Foundation through the New Cornerstone Investigator Program and the Xplorer Prize, as well as by the Hong Kong RGC grant 14300122. 

\section{Pure dimensionality}\label{sec:Pure Dimensionality}


We first prove equidimensionality.

\begin{lemma}\label{lemma:Lusztigvarietypuredimensionsemisimple}
Let $C$ be a regular $\d$-conjugacy class of $G$. For any $w \in W$, we have 

\begin{enumerate}
    \item $C \cap B \dot w B$ is of pure dimension $\dim B + \ell(w) - \dim T^\d$,
    \item for any $h \in C$, $\CY_{w, h, \d}$ is of pure dimension $\ell(w)$.
\end{enumerate}
\end{lemma}

\begin{remark}
    The result for regular semisimple conjugacy classes is due to Lusztig \cite[Lemma 1.1]{lusztig197912}. 
\end{remark}

\begin{proof}
Let $h \in C$. By \cite[\S5.5]{he2023affine}, we have $\dim \CY_{w, h, \d}= \ell(w)$. We have a diagram
\begin{equation}\label{eq:diagram}
    \xymatrix{G/B & G \ar[l]_-{\pi} \ar[r]^-{\rho} & C,}    
\end{equation}
where $\pi\colon G \to G/B$ is the natural projection map and $\rho\colon G \to C\colon g\mapsto g \cdot_\d h$. Note that $\rho$ is a fiber bundle with fiber isomorphic to $Z_{G, \d}(h)$. By definition, $C \cap B \dot w B=\rho(\pi^{-1}(\CY_{w, h, \d}))$, so 
\begin{equation*}
    \dim(C \cap B \dot w B)=\dim(\rho(\pi^{-1}(\CY_{w, h, \d})) = \ell(w) + \dim B - \dim T^\d.
\end{equation*}
Each irreducible component of $C \cap B \dot w B$ has dimension at least
\begin{align*}
\dim(C) + \dim(B \dot w B)- \dim G
&= 
\dim G - \dim T^\d + \dim B + \ell(w) - \dim G
\\
&=
\ell(w) + \dim B - \dim T^\d.
\end{align*}
Hence $C \cap B \dot w B$ has pure dimension $\dim B + \ell(w) - \dim T^\d$. Moreover, using the diagram \eqref{eq:diagram}, we deduce that $\CY_{w, h, \d}$ is of pure dimension $\ell(w)$.
\end{proof}



\section{Iwahori--Hecke algebra and counting rational points}\label{sec:Hecke algebra}

\begin{definition}
    The \emph{Iwahori--Hecke algebra $\CH$ associated with $W$} is the $\BZ[t]$-algebra with basis $\{T_w\colon w \in W\}$ subject to the relations
\begin{align*}
    &T_wT_{w'} = T_{ww'} &&\text{for $w,w'\in W$ with $\ell(ww') = \ell(w)+\ell(w') = \ell(ww')$,}
    \\
    &(T_s-t)(T_s+1) = 0 &&\text{for $s\in \BS$ (the `quadratic relations').}
\end{align*}
    For $A \in \CH$ and $w \in W$, we denote by $[A:T_{w}]$ the coefficient $T_{w}$ in $A$ with respect to the standard basis $\{T_{w'}\colon w' \in W\}$.
\end{definition}

\begin{lemma}\label{lemma:DigneMichelHeckecoefficients}
Suppose $\supp_\delta(w) = \BS$. Then
\begin{enumerate}
\item $U^- \cap B \dot w B$ is irreducible of dimension $\ell(w)$,
\item $\dim(\dot w'^{-1}U^-\delta(\dot w') \cap B \dot w B) < \ell(w)+\ell(w')$ for $w' \in W$ with $w' \neq 1$. 
\end{enumerate}
\end{lemma}

\begin{proof}

    Suppose first that $\charac \sfK = p > 0$. The characteristic 0 case will be deduced later from the result for positive characteristic. 
    As $\sfK$ is algebraically closed, $\BG$ has a Borel subgroup and maximal torus defined over $\BZ$. By conjugating the action of $\delta$ with an inner automorphism, we may assume that $B$, $T$, $N_G(T)$, $U$ and $U^-$ are $\delta$-stable and defined over $\BZ$. Then $B\dot wB$ is defined over $\BZ$. Moreover, $\dot w' U^-$ is defined over $\BZ$ as $\dot w' U^- \delta(\dot w') = \dot w' U^- \dot w''$ for some $w'' \in W$. Thus to prove the $\charac \sfK = p$ case, we may assume that $\sfK = \bar\BF_p$.
    
    In this case, we have a (geometric) Frobenius morphism $F$ on $\BG$ corresponding to $\BF_q$. By picking $q$ large enough, we may assume that $F(\d)=\d$. 
    Furthermore, we can pick an $F$-stable Borel subgroup and assume it is also $\delta$-stable by appropriately conjugating $
    \delta$ with an inner automorphism, thus we may assume that $B$, $T$, $N_G(T)$, $U$ and $U^-$ are all $F$-stable and $\delta$-stable.

    We have the following general facts (i.e. for any field $\sfK$):
    \begin{align*}
        (\dot w')^{-1}U^-\delta(\dot w') 
        &= (\dot w')^{-1}\delta(\dot w') (\delta(\dot w'^{-1}) U^- \d(\dot w')) 
        \\ 
        &\cong
        (\dot w')^{-1}\delta(\dot w')(\delta(\dot w'^{-1}) U^- \d(\dot w')  \cap  U^-) \times (\delta(\dot w'^{-1}) U^- \d(\dot w')  \cap U),
    \end{align*}
    and as $B \dot w B$ is stable under right multiplication with $\delta((\dot w')^{-1}) U^- \d(\dot w')$, we have 
    \begin{align}\label{eq:decompositionUBwB}
        &(\dot w')^{-1}U^-\delta(\dot w') \cap B \dot w B 
        \\
        \cong\,&((\dot w')^{-1}U^-\delta(\dot w')\cap (\dot w')^{-1}\delta(\dot w') U^- \cap B \dot w B)\times(\delta((\dot w')^{-1}) U^- \d(\dot w') \cap U).
    \end{align}
    By \cite[Theorem 2.6b]{kawanaka1975unipotent}, we have $|((\dot w')^{-1}U^-\delta(\dot w')\cap (\dot w')^{-1}\delta(\dot w') U^- \cap B \dot w B)^F| = [T_wT_{\delta(w')^{-1}w_0} : T_{(w')^{-1}w_0}]$, and together with \eqref{eq:decompositionUBwB}, this gives 
    \begin{equation}\label{eq:kawanaka}
        |((\dot w')^{-1}U^-\delta(\dot w') \cap B\dot w B)^F| = [T_wT_{\delta(w')^{-1} w_0} : T_{(w')^{-1}w_0}] q^{\ell(w')}.
    \end{equation}
    By the quadratic relations in the Iwahori--Hecke algebra, $[T_w T_{w_0}: T_{w_0}]$ is a monic polynomial in $q$ of degree $\ell(w)$. By \cite[Lemma 8.6]{digne2006endomorphisms}, $\sum_{w'\in W} [T_wT_{\delta(w')^{-1} w_0} : T_{(w')^{-1}w_0}]$ is also a monic polynomial of degree $\ell(w)$, so $[T_wT_{\delta(w')^{-1} w_0} : T_{(w')^{-1}w_0}]$ is a polynomial in $q$ of degree less than $\ell(w)$ for all $w' \in W$ with $w' \neq 1$. 
    So by \eqref{eq:kawanaka} for $w'=1$, it follows that $|(U^- \cap B \dot w B)^F|$ is a monic polynomial in $q$ of degree $\ell(w)$, so $U^- \cap B \dot w B$ is irreducible of dimension $\ell(w)$. Similarly, by \eqref{eq:kawanaka} for $w'\neq1$, it follows that $\dim((\dot w')^{-1}U^-\delta(\dot w') \cap B\dot w B) < \ell(w)+\ell(w')$. 
    The Lemma thus holds for the $\charac \sfK > 0$ case. 


    Suppose $\charac \sfK = 0$. Similar as before, 
    we may assume that $B$, $T$, $N_G(T)$, $U$, $U^-$ are $\delta$-stable and defined over $\BZ$. For $w'\in W$, the reduction of $\dot w'^{-1} U^- \delta(\dot w') \cap B\dot wB$ from $\BZ$ to $\BZ/p\BZ$ does not change the number of irreducible components, nor the dimension. Hence the result follows from the already proved $\charac \sfK > 0$ case.
\end{proof}

\begin{proposition}\label{prop:irr}
    Let $C$ be a regular $\d$-conjugacy class of $G$ and $w \in W$ with $\text{supp}_{\d}(w)=\BS$. Then $C \cap B \dot w B$ is irreducible.
\end{proposition}

\begin{proof}
    Let $B^-=T U^-$ be the Borel subgroup opposite to $B$. Then $B^-$ is $\d$-stable. By Steinberg's theorem \cite[Lemma 7.3]{steinberg1968endomorphisms}, any $\d$-conjugacy class of $G$ contains an element of $B^-$. Let $h \in B^- \cap C$. 
    We have 
    \begin{equation}\label{eq:Cdecomposition}
        C=G\cdot_\d h=\bigcup_{w' \in W} (B \dot w' B^-) \cdot_\d h.
    \end{equation}
    Note that $B \dot w' B^-=(U \cap (\dot w') U (\dot w') \i) \dot w' B^-$ and that $B \dot w B$ is stable under the $\d$-conjugation action of 
    $U \cap (\dot w' U (\dot w')^{-1})$, hence
    \begin{equation*}
        (B \dot w' B^-) \cdot_\d h \cap B \dot w B
        =
        (U \cap \dot w' U (\dot w') \i) 
        \cdot_\d 
        (\dot w' (B^- \cdot_\d h) \delta(\dot w'^{-1}) \cap B \dot w B).
    \end{equation*}
    In particular, $\dim ((B \dot w' B^-) \cdot_\d h \cap B \dot w B) \le \dim (\dot w' (B^- \cdot_\d h) \d(\dot w'^{-1}) \cap B \dot w B)+\ell(w_0 w')$. 

    Let $t \in T$ and $u \in U^-$ such that $h=tu$. Then $B^- \cdot_\d h \subset t T_1 U^-$, where $T_1=\{(t' \d(t') \i\colon t' \in T\}$, which is irreducible of dimension $\dim T-\dim T^\d$, so
    \begin{equation*}
        \dot w' (B^- \cdot_\d h) \delta(\dot w'^{-1}) \cap B \dot w B 
        \subseteq 
        t T_1 (\dot w' U^- \d(\dot w'^{-1}) \cap B \dot w B)
        \cong
        \dot w' U^- \d(\dot w'^{-1}) \cap B \dot w B.
    \end{equation*} 
    If $w' \neq 1$, then by Lemma \ref{lemma:DigneMichelHeckecoefficients}(2), we have
    \begin{align}
        &\dim ((B \dot w' B^-) \cdot_\d h \cap B \dot w B) 
        \\
        \le& \dim (\dot w' (B^- \cdot_\d h) \d(\dot w'^{-1}) \cap B \dot w B)+\ell(w_0 w') 
        \\ 
        \le& \dim(\dot w' U^- \d(\dot w'^{-1}) \cap B \dot w B)+\dim T-\dim T^\d+\ell(w_0 w') 
        \\ 
        <&\ell(w)+\ell(w_0)+\dim T-\dim T^\d 
        \\ 
        =&\dim B+\ell(w)-\dim T^\d. \label{eq:dimBw'B^-}
    \end{align}
    Since $C \cap B \dot w B=\bigcup_{w' \in W} (B \dot w' B^-) \cdot_\d h \cap B \dot w B$, it follows from Lemma \ref{lemma:Lusztigvarietypuredimensionsemisimple}(1) that 
    \begin{equation}\label{eq:dimBBhBwB}
        \dim((B B^-) \cdot_\d h \cap B \dot w B) = \dim B+\ell(w)-\dim T^\d.
    \end{equation}
        We also have 
    \begin{equation*}
        (B B^-) \cdot_{\delta} h \cap B \dot w B
        =
        U \cdot_{\delta} (B^- \cdot_{\delta} h \cap B \dot w B) 
        \subseteq 
        U \cdot_{\delta} (t T_1 (U^- \cap B \dot w B)),
    \end{equation*} 
    so $\dim(U \cdot_\d (t T_1 (U^- \cap B \dot w B))) \geq \dim((B B^-) \cdot_\d h \cap B \dot w B) = \dim B+\ell(w)-\dim T^\d$.    
    By Lemma \ref{lemma:DigneMichelHeckecoefficients}(1), $U^- \cap B \dot w B$ is irreducible of dimension $\ell(w)$, so $U \cdot_\d (t T_1 (U^- \cap B \dot w B))$ is irreducible of dimension $\dim U+\dim T_1+\ell(w)=\dim B+\ell(w)-\dim T^\d$.
    %
    Thus $(B B^-) \cdot_\d h \cap B \dot w B$ is contained in an irreducible variety of the same dimension, and is therefore itself irreducible. 
    Together with \eqref{eq:dimBw'B^-}, \eqref{eq:dimBBhBwB}, and the equidimensionality of $C \cap B \dot w B$, it follows that $C \cap B \dot w B$ is irreducible.
\end{proof}

\section{Main theorem}\label{sec:rational}

\begin{lemma}\label{lemma:lusztigvarietyrationalpointskawanaka}
Let $p$ be a prime number.
Suppose $\sfK = \bar\BF_p$, $\BG$ is defined over $\BF_q$ where $q$ is some power of $p$, and let $F$ be the corresponding Frobenius morphism. Suppose $F$ acts trivially on $\BG$, $F(\d)=\d$, and $B$ and $T$ are $F$-stable and $\delta$-stable.
Let $h \in G^F$ and $w \in W$. Then
$$
|\CY_{w, h, 
\d}^F|
=
|G^F||B^F|^{-1}|G^F\cdot_\delta h \cap (B \dot w B)^F||G^F\cdot_\delta h|^{-1}.
$$
\end{lemma}

\begin{remark}
    The case where $\d$ is the identity map is established by Kawanaka in \cite[Lemma 3.6]{kawanaka1975unipotent}. The general case can be proved the same way. We include the proof for the convenience of the readers. 
\end{remark}

\begin{proof}
Since $B$ is connected, we have $(G/B)^F = G^F/B^F$. Let 
\begin{align*}
A_w 
&= \{(gB,g')\in G^F/B^F \times (G^F\cdot_\delta h) \colon gB \in \CY_{w, g', \d}\}
\\
&= \{(gB,g')\in G^F/B^F \times (G^F\cdot_\delta h) \colon g^{-1} g'\delta(g)\in B \dot w B\}.
\end{align*}
We have a natural projection map from $A_w$ to $G^F/B^F$ which is surjective and whose fibers are in bijection with $G^F \cdot_\d h \cap B \dot w B=G^F \cdot_\d h \cap (B \dot w B)^F$, so
\begin{equation}
    \label{eq:Aw1}
    |A_w| 
    = |G^F||B^F|^{-1}||G^F\cdot_\delta h \cap (B \dot w B)^F|.
\end{equation}
For $g \in G^F$, we have $\CY_{w, g \cdot_\d h, \d}=g \CY_{w, h, \d}$, hence $\CY_{w, g \cdot_\d h, \d}^F=g \CY_{w, h, \d}^F$. We also have a natural projection map from $A_w$ to $G^F \cdot_\d h$. This map is surjective and has fibers isomorphic to $\CY_{w, h, \d}^F$, so we have 
\begin{equation}
    \label{eq:Aw2}
    |A_w| = |G^F\cdot_\delta h||\CY_{w, h, \d}^F|.
\end{equation} 
The Lemma follows from \eqref{eq:Aw1} and \eqref{eq:Aw2}.
\end{proof}

\begin{theorem}\label{theorem:lusztigvarietyirreduciblearbitrarycharacteristic}
    Suppose $\supp_\delta(w) =\BS$ and let $h$ be a $\d$-regular element in $G$. Then $\CY_{w, h, \d}$ is irreducible.
\end{theorem}
\begin{proof}
    Suppose first that $\sfK = \bar \BF_p$ for a prime number $p$. We will later show how the general case reduces to this case. Suppose $\BG$ is defined over $\BF_q$ for some power $q$ of $p$ and let $F$ be the corresponding Frobenius morphism.
    By picking $q$ sufficiently large, we may assume that $h \in G^F$  and that $\BG$ is split over $\BF_q$, so that $F$ acts trivially on $W$. As in the proof of Lemma \ref{lemma:DigneMichelHeckecoefficients}, we may also assume that
    \begin{equation}\label{eq:assumptions}
        F(\d)=\d, F(B)=\delta(B)=B, F(T)=\delta(T)=T.
    \end{equation}
    Let $C=G \cdot_\d h$. As $Z_{G,\delta}(h)$ may be disconnected, there may be more than one $G^F$-orbit in $C^F$. By \cite[Proposition 4.2.14]{digne2020representations}, the $G^F$-orbits on $C^F \cong (G/Z_{G,\delta}(h))^F$ are in natural bijection with the $F$-conjugacy classes on $Z_{G,\delta}(h)/Z_{G,\delta}(h)^\circ$. Let $h_1, \ldots, h_k$ be the representatives of the $G^F$-orbits in $C^F$. Then 
    \begin{align}\label{eq:sumCBwB}
        \sum_{i=1}^k |G^F \cdot_\d h_i \cap B \dot w B|=|C^F \cap B \dot w B| = |(C \cap B \dot w B)^F|.   
    \end{align}
    By Lemma \ref{lemma:Lusztigvarietypuredimensionsemisimple}(1) and Proposition \ref{prop:irr}, $|C^F \cap B \dot w B| $ is a monic polynomial over $q$ of degree $\dim B+\ell(w)-\dim T^\d$. Thus $|G^F| |C^F \cap B \dot w B|$ is a monic polynomial over $q$ of degree $\dim B+\dim G+\ell(w)-\dim T^\d=\dim B+\dim C+\ell(w)$.

    Let $i \in \{1,\dots,k\}$. Note that $G^F \cdot_\d h_i \cong G^F/Z_{G,\delta}(h_i)^F$ and $Z_{G,\delta}(h_i) \cong Z_{G,\delta}(h)$ is a subgroup of $G$ of dimension $\dim G-\dim C$, so $|G^F \cdot_\d h_i|$ is a polynomial in $q$ of degree $\dim C$ with rational coefficients. By Lemma \ref{lemma:Lusztigvarietypuredimensionsemisimple}, $|\CY_{w, h_i, \d}^F|$ is a polynomial in $q$ of degree $\ell(w)$ with integer coefficients. Let $a_i \in \BQ_{>0}$ and $b_i \in \BN$ be the leading coefficient of $|G^F \cdot_\d h_i|$ and $|\CY_{w, h_i, \d}^F|$, respectively. Since $C$ is a single $G$-orbit, it is irreducible. Hence $\sum_i a_i=1$, as $\sum_i |G^F \cdot_\d h_i|=|C^F|$.
    Note that $|\CY_{w, h_i, \d}^F| |B^F| |G^F \cdot_\d h_i|$ is a polynomial over $q$ 
    with leading coefficient $a_i b_i$.
    By Lemma \ref{lemma:lusztigvarietyrationalpointskawanaka}, we have
    \begin{align}\label{eq:YBGh}
        |\CY_{w, h_i, \d}^F| |B^F| |G^F \cdot_\d h_i|
        =
        |G^F|
        |G^F\cdot_\delta h_i \cap B \dot w B|.
    \end{align}
    In particular, we have $\sum_{j=1}^k |\CY_{w, h_j, \d}^F| |B^F| |G^F \cdot_\d h_j| = |G^F||C^F\cap B\dot w B|$, which is a monic polynomial over $q$, as we deduced earlier that $|C^F\cap B\dot w B|$ is monic. Hence we have $\sum_{j=1}^k a_j b_j=1$. Since $b_j \in \BN$ and $a_j>0$ for $j=1,\dots,k$, we have $1=\sum_{j=1}^k a_j b_j \geq \sum_{j=1}^k a_j=1$. Therefore $b_i =1$, so $\CY_{w, h_i, \d} \cong \CY_{w,h,\d}$ is irreducible. 

    Suppose $\sfK$ is any algebraically closed field of characteristic $p$.
    The reduction to the $\bar\BF_p$ case follows the same reasoning as in \cite[Appendix B]{tsai2020components}. There exists a subfield $\sfk$ of $\sfK$ that is finitely generated over its prime field such that $h \in G(\sfk)$. 
    Let $R$ be a subring of $\sfk$ of finite type over $\BF_p$ such that $\sfk$ is its quotient field and $h \in \BG(R)$. 
    Note that $\Spec(R)$ has generic point $\Spec(\sfk)$.
    We may extend $R$ (i.e. invert a finite number of elements in $R$) so that $Y_{w,h,\delta}$ is defined over $\Spec(R)$.
    Let $\pi \colon \CY_{w,h,\delta} \to \Spec(R)$ be the natural map. Let $\ell \neq p$ be a prime and let $\underline{\BQ}_\ell$ be the constant sheaf associated to $\BQ_\ell$. Since $\Spec(R) \to \BF_q$ is of finite type it follows that $R^{2\ell(w)}\pi_!\underline{\BQ}_\ell$ (here $\pi_!$ denotes the proper pushforward) is a constructible sheaf over $\Spec(R)$. 
    Since the theorem is proved for the case $\sfK = \bar\BF_p$, it follows that the stalks of $R^{2\ell(w)}\pi_!\underline{\BQ}_\ell$ at the closed points of $\Spec(R)$ are $1$-dimensional. By the definition of constructibility, and since closed points are Zariski dense, the stalks at all points must be of dimension $1$, particularly at the generic point $\Spec(\sfk)$. By proper base change, the top $\ell$-adic \'{E}tale cohomology with compact support $H_c^{2\ell(w)}(Y_{w,h,\delta},\BQ_\ell)$ of $Y_{w,h,\delta}$ has dimension one. A general fact is that this top cohomology has a basis indexed by the irreducible components of $Y_{w,h,\delta}$, so $Y_{w,h,\delta}$ is irreducible.

    Suppose $\charac\sfK=0$. We may assume as in the proof of Lemma \ref{lemma:DigneMichelHeckecoefficients} that $G$, $B$, $T$ and $N_G(T)$ are defined over $\BZ$ and $\delta$-stable. Then $G/B$ is defined over $\BZ$, and since $B\dot w B$ is independent of the representative of $\dot w$ in $N_G(T)$, it is also defined over $\BZ$. There exists a finitely generated ring extension $S$ of $\BZ$ such that $h \in \BG(S)$. 
    Recall $\pi$ and $\rho$ from \eqref{eq:diagram}.
    Then $Y_{w,h,\delta} = \pi\circ\rho^{-1}(B\dot w B)$ is defined over $S$.
    For a prime number $p$, denote by $\CY_{w,h,\delta}^p$ the reduction of $Y_{w,h,\delta}$ to $S/pS$. 
    By the Bertini--Noether theorem \cite[Proposition 10.4.2]{fried2005field}, it holds for a sufficiently large prime $p$ that $\CY_{w,h,\delta}^p$ has the same number of irreducible components as $\CY_{w,h,\delta}$. We already proved that $\CY_{w,h,\delta}^p \times_{\Spec(S/pS)} \Spec{(\overline{S/pS})}$ is irreducible, so the result follows.
\end{proof}

\bibliographystyle{alpha}
\bibliography{main}

\begin{thebibliography}{Kaw75}

\bibitem[DM06]{digne2006endomorphisms}
Fran{\c{c}}ois Digne and Jean Michel.
\newblock Endomorphisms of {D}eligne-{L}usztig varieties.
\newblock {\em Nagoya Math. J.}, 183:35--103, 2006.

\bibitem[DM20]{digne2020representations}
Fran\c{c}ois Digne and Jean Michel.
\newblock {\em Representations of finite groups of {L}ie type}, volume~95 of {\em London Mathematical Society Student Texts}.
\newblock Cambridge University Press, Cambridge, 2020.
\newblock Second edition of [ 1118841].

\bibitem[FJ05]{fried2005field}
Michael Fried and Moshe Jarden.
\newblock {\em Field arithmetic}, volume~11 of {\em A Series of Modern Surveys in Mathematics}.
\newblock Springer-Verlag, Berlin, second edition, 2005.

\bibitem[He23]{he2023affine}
Xuhua He.
\newblock On affine {L}usztig varieties.
\newblock {\em arXiv preprint arXiv:2302.03203. To appear in Ann. Sci. \'Ec. Norm. Sup\'er.}, 2023.

\bibitem[He25]{he25}
Xuhua He.
\newblock Irreducible components of affine {L}usztig varieties.
\newblock {\em preprint}, 2025.

\bibitem[Kaw75]{kawanaka1975unipotent}
Noriaki Kawanaka.
\newblock Unipotent elements and characters of finite {C}hevalley groups.
\newblock {\em Proc. Japan Acad.}, 51:156--158, 1975.

\bibitem[Kim20]{kim2020homology}
Dongkwan Kim.
\newblock Homology class of a {D}eligne--{L}usztig variety and its analogs.
\newblock {\em Int. Math. Res. Not.}, 2020(4):1246--1280, 2020.

\bibitem[Lus79]{lusztig197912}
George Lusztig.
\newblock On the reflection representation of a finite {C}hevalley group.
\newblock {\em Representation Theory of Lie Groups}, 34:325, 1979.

\bibitem[Lus85]{lusztig1985character1}
George Lusztig.
\newblock Character sheaves {I}.
\newblock {\em Adv. Math.}, 56(3):193--237, 1985.

\bibitem[Ste68]{steinberg1968endomorphisms}
Robert Steinberg.
\newblock {\em Endomorphisms of linear algebraic groups}.
\newblock Memoirs of the American Mathematical Society, No. 80. American Mathematical Society, Providence, RI, 1968.

\bibitem[Tsa20]{tsai2020components}
Cheng-Chiang Tsai.
\newblock Components of affine springer fibers.
\newblock {\em In. Math. Res. Not.}, 2020(6):1882--1919, 2020.

\end{thebibliography}

\end{document}